\documentclass{amsart}

\usepackage[all]{xy}
\usepackage{amsmath,amssymb}
\usepackage{amsthm}
\usepackage{enumerate}
\usepackage[dvips]{graphicx}
\usepackage{txfonts}
\usepackage{url}

\usepackage{tikz}

\theoremstyle{plain}
\newtheorem{theorem}{Theorem}[section]
\newtheorem{corollary}[theorem]{Corollary}
\newtheorem{lemma}[theorem]{Lemma}
\newtheorem{proposition}[theorem]{Proposition}

\theoremstyle{definition}
\newtheorem{definition}[theorem]{Definition}
\newtheorem{example}[theorem]{Example}
\newtheorem{remark}[theorem]{Remark}

\newcommand{\ses}{\searrow \!\!\! \searrow}
\def\co{\colon\thinspace}

\newcommand{\op}{\mathrm{op}}
\newcommand{\ob}{\mathrm{ob}}
\newcommand{\mor}{\mathrm{mor}}
\newcommand{\id}{\mathrm{id}}

\newcommand{\sd}{\mathrm{sd}}

\newcommand{\A}{\mathcal{A}}
\newcommand{\B}{\mathcal{B}}

\newcommand{\cS}{\mathcal{S}}
\newcommand{\C}{\mathcal{C}}

\newcommand{\bs}{\backslash}

\begin{document}
	
	\author{Kohei Tanaka}
	\address{Institute of Social Sciences, School of Humanities and Social Sciences, Academic Assembly, Shinshu University, 3-1-1 Asahi, Matsumoto, Nagano 390-8621, Japan.}
	\email{tanaka@shinshu-u.ac.jp}
	\title{Strong homotopy types of acyclic categories and $\Delta$-complexes}

	\maketitle
	\thispagestyle{empty}

	

	\begin{abstract}
		We extend the homotopy theories based on point reduction for finite spaces and simplicial complexes to finite acyclic categories and $\Delta$-complexes, respectively. The functors of classifying spaces and face posets are compatible with these homotopy theories. 
		In contrast with the classical settings of finite spaces and simplicial complexes, the universality of morphisms and simplices plays a central role in this paper.	
	\end{abstract}

	\section{Introduction}
	
	The homotopy theory for finite $T_0$ topological spaces (finite spaces), which was developed by Stong \cite{Sto66}, is an important tool in combinatorial algebraic topology \cite{Bar11}. It is based on successive removal of points and can be described in purely combinatorial terms.
	For a finite space, Stong found a minimal subspace with the same homotopy type as the original space. It is called a {\em core} and is uniquely determined up to isomorphism.
	A similar notion is known as the strong homotopy theory for finite simplicial complexes and is based on vertex reduction. Barmak and Minian studied the relation between these homotopy theories in \cite{BM12}. They showed that order complexes and face posets are compatible with the homotopy theories of finite spaces and the strong homotopy theory of simplicial complexes.
	
	This paper develops Barmak and Minian's work to acyclic categories and $\Delta$-complexes. We are only concerned with finite acyclic categories and finite $\Delta$-complexes in this paper.
	A finite space can be regarded as a finite partially ordered set (poset), and a category with at most one morphism for each pair of objects and no loops of morphisms. An acyclic category is a generalization of a poset from the categorical viewpoint. Namely, it is a category without loops of morphisms allowed to have multiple morphisms. 
	Let us recall the strong homotopy theory for categories based on natural transformations \cite{Lee73}, \cite{Hof74}. As in the case of finite spaces, removing objects of an acyclic category $\A$ plays an important role for the strong homotopy theory. There is a unique, up to isomorphism, minimal subcategory of $\A$ that has the same strong homotopy type as $\A$.

	Moreover, we consider generalizations of simplicial complexes that are called {\em unordered $\Delta$-complexes}. An ordered $\Delta$-complex ($\Delta$-set \cite{RS71}, trisp \cite{Koz08}, semi-simplicial complex \cite{EZ50}) is defined as a simplicial set without degeneracy maps. It is equipped with face maps satisfying the simplicial relation and we distinguish the $i$-th vertex of an $n$-simplex for $i=0,1,\ldots,n$. An ordered simplicial complex can be naturally considered an ordered $\Delta$-complex.

	If we are not interested in the ordering of vertices, we obtain the weaker notion of unordered $\Delta$-complexes. Even though this notion was introduced in Appendix of Hatcher's book \cite{Hat02} as a special type of a regular cell complex; we describe it in combinatorial terms herein.  A standard (unordered) simplicial complex is a special case of an unordered $\Delta$-complex. As this paper mainly focus on unordered $\Delta$-complexes rather than ordered $\Delta$-complexes, an unordered $\Delta$-complex will be referred to as a $\Delta$-complex, in accordance with the case of simplicial complexes.
	As in the case of acyclic categories, we establish the strong homotopy theory for $\Delta$-complexes and characterize it by vertex reduction. Furthermore, a minimal subcomplex of a $\Delta$-complex $X$ with the same strong homotopy type as $X$ is uniquely determined up to isomorphism.
	
	For an acyclic category $\A$, the {\em classifying space} $B\A$ is a $\Delta$-complex consisting of sequences of composable morphisms. When $\A$ is a poset, the classifying space becomes a simplicial complex and is occasionally called the {\em order complex}. Moreover, for a $\Delta$-complex $X$, the {\em face poset} $\chi(X)$ consists of the simplices of $X$ and is ordered by inclusion.
	We study the relation between these functors and strong homotopy types. In particular, we focus on strong collapsibility (strong homotopy type of a single point) of acyclic categories and $\Delta$-complexes. The functors of classifying spaces and face posets are completely compatible with strong collapsibility.

	\begin{theorem}[Theorem \ref{B_collapse}, \ref{chi_collapse}]
	Let $\A$ be an acyclic category and let $X$ be a $\Delta$-complex.
	\begin{itemize}
		\item $\A$ is strongly collapsible if and only if $B\A$ is strongly collapsible.
		\item $X$ is strongly collapsible if and only if $\chi(X)$ is strongly collapsible.
	\end{itemize}
	\end{theorem}

	The remainder of this paper is organized as follows: In Section 2, we recall some basic notions and properties related to acyclic categories. Using the universality of morphisms, we introduce reducible objects of an acyclic category without changing the strong homotopy type. They are called {\em beat objects}. 
	We obtain a minimal subcategory by successively deleting all beat objects. This is determined uniquely up to isomorphism, regardless of the order in which the points are removed.

	Section 3 includes a similar discussion for $\Delta$-complexes. Using the universality of simplices, we introduce reducible vertices without changing the strong homotopy type. They are called {\em dominated vertices}. 
	A minimal $\Delta$-subcomplex is uniquely obtained, up to isomorphism, by successively removing all dominated vertices.

	In Section 4, we study the homotopy type of classifying spaces and face posets. The barycentric subdivisions of acyclic categories and $\Delta$-complexes can be defined using classifying spaces and face posets. The barycentric subdivision of an acyclic category becomes a poset and the barycentric subdivision of a $\Delta$-complex becomes a simplicial complex. 
This relates our strong homotopy theory for acyclic categories and $\Delta$-complex to the classical one for finite spaces (posets) and simplicial complexes.

	 \begin{figure}[ht]
  \begin{center}
   \begin{tikzpicture}
	\draw [fill,lightgray] (1,0) circle (1);
	\draw [fill] (0,0) circle (2pt);
	\draw [fill] (1,0) circle (2pt);
	\draw [fill] (2,0) circle (2pt);
	\draw (0,0) -- (1,0) -- (2,0);
	\draw (1,0) circle (1); 
	\draw (-0.2,0.15) node {$v$};
	\draw (1.2,0.2) node {$v'$};

	\draw [->]  (2.6,0.25) -- (3.1,-0.25);
	\draw [->]  (2.8,0.25) -- (3.3,-0.25);
	
	\draw [fill] (4,0) circle (2pt);
	\draw [fill] (5,0) circle (2pt);
	\draw (4,0) -- (5,0);
	\draw (4.2,0.2) node {$v'$};  
  \end{tikzpicture}  
  \end{center}
  \caption{Elementary strong collapse on $2$-disc associated with a dominated vertex $v$}
  \label{figure1}
 \end{figure}
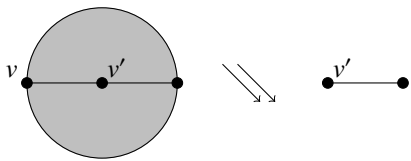

	This paper provides a natural extension of the classical notions and Barmak-Minian's work on finite spaces and simplicial complexes. However, it is essential that we pay attention to the universality of morphisms and simplices. A poset has a unique morphism for each pair of objects (if it exists) and a simplex of a simplicial complex is uniquely determined by its vertices. This is in contrast with the classical strong homotopy theory.

	\section{Strong homotopy types of acyclic categories}

	For a category $\A$, let $\ob(\A)$ denote the set of objects and $\mor(\A)$ the set of morphisms of $\A$. Moreover, for a pair of objects $x$ and $y$, the set of morphisms starting at $x$ and ending at $y$ is denoted by $\A(x,y)$.
	This paper is mainly concerned with acyclic categories without loops of morphisms.

	\begin{definition}
	A category $\A$ is called {acyclic} if it satisfies the following two conditions:
		\begin{enumerate}
			\item $\A(x,y)=\emptyset$ if $\A(y,x) \neq \emptyset$ for each pair of objects $x \neq y$.
			\item The set of endomorphisms $\A(x,x)$ consists of the identity only, for each object $x$.
		\end{enumerate}
	\end{definition}

	An important example of acyclic categories is the class of posets.
A poset $P$ can be regarded as an acyclic category with at most one morphism: $P(x,y)$ consists of one morphism when $x \geq y$ and $P(x,y) = \emptyset$ otherwise.
	Furthermore, a finite poset can be viewed as a finite topological space. The homotopy theory of finite spaces was developed by Stong \cite{Sto66}. 
	He showed that two maps $f,g$ between finite spaces $P,Q$ are homotopic if and only if there exists a finite sequence of natural transformations connecting $f$ and $g$.
	From this viewpoint, we may consider a homotopy relation of general functors based on natural transformations \cite{Lee73}, \cite{Hof74}. Here, we use the terminology of Minian's paper \cite{Min02}.

	\begin{definition}\label{acyclic_homotopy}
	We call two functors $F, G \co \A \to \B$ {\em homotopic} and use the notation $F \simeq G$ if there are functors $H_i \co \A \to \B$ for $i=0, \ldots,n$ with natural transformations: 
	\[
	F = H_0 \Rightarrow H_1 \Leftarrow H_2 \Rightarrow \ldots \Rightarrow H_n=G.
	\]
	A functor $F \co \A \to \B$ is a {\em strong homotopy equivalence} if there exists a functor $G \co \B \to \A$ satisfying $G \circ F \simeq 1_{\A}$ and $F \circ G \simeq 1_{\B}$. If there exists a strong homotopy equivalence between $\A$ and $\B$, we say that these categories are {\em strongly homotopy equivalent} and use the notation $\A \simeq \B$.
	\end{definition}

	Stong found homotopically reducible points of finite spaces and called them {\em beat points} (the original term was {\em linear points}). We can naturally extend this notion to the setting of categories. 
	For a category $\A$ and an object $x$ of $\A$, the over category $\A/x$ is defined as follows: $\ob(\A/x) = \{ f \in \mor(\A) \mid f \co y \to x\}$ and $(\A/x)(f,g) = \{h \in \mor(\A) \mid g \circ h=f\}$. This category has as terminal object the identity $\id_x$.
The full subcategory $\hat{\A}/x$ of the over category $\A/x$ is defined by removing the identity $\id_x$, i.e., $\ob(\hat{\A}/x) = \ob(\A/x) \bs \{id_x\}$.

	\begin{definition}\label{beat_object}
	An object $x$ of an acyclic category $\A$ is called a {\em down beat object} if $\hat{\A}/x$ has a terminal object $f \co y \to x$. Namely, any nontrivial morphism $g \co z \to x$ is uniquely factored through $f$ in $\A$, that is, there exists a unique morphism $\tilde{g} \co z \to y$ with $g=f \circ \tilde{g}$. We call $f$ the {\em down beat morphism} that is associated with $x$. 
	Dually, we can define the notions of {\em up beat objects} and {\em up beat morphisms} by reversing the directions of morphisms. That is, an up beat object (resp., morphism) is defined as a down beat object (resp., morphism) in the opposite category $\A^{\op}$. 
	We say that an object is a {\em beat object} if it is either a down beat object or an up beat object.
	\end{definition}

	\begin{remark}
	If $\A$ is a finite poset (finite space), the category $\A/x$ over an object $x$ is called the {\em prime ideal} of $x$. The notion of beat objects of $\A$ coincides with the classical notion of beat points. 
	\end{remark}

	 The down beat morphism $f \co y \to x$ associated with a down beat object $x$ of $\A$ induces a functor $F \co \A \to \A$ defined by $Fa=a$ for $a \neq x$, and $Fx=y$ on objects. For a morphism $g \co a \to b$ with $g \neq \id_x$, we define 
	\[
	Fg=
		\begin{cases}
			g & \textrm{if $a \neq x, b \neq x$},\\
			g \circ f & \textrm{if $a=x,  b \neq x$},\\ 
			\tilde{g} & \textrm{if $a \neq x, b=x$}, 
		\end{cases}
	\]
	and $F(\id_x)=\id_y$.  
	We notice that this functor has the following two properties:
	\begin{itemize}
		\item It is a projection (retraction), i.e., $F^2 = F \co \A \to \A$.
		\item It is equipped with a natural transformation $t \co F \Rightarrow 1_A$ that is given by $t_a=\id_a$ for $a \neq x$ and $t_x=f$.
	\end{itemize}
	In the homotopy theory of finite posets, such functors (order-preserving maps) are known as {\em closure operators} (see Section 13.2 of \cite{Koz08}). Let us extend the notion of closure operators to functors.

	\begin{definition}
	A functor $F \co \A \to \A$ on an acyclic category $\A$ is called a {\em descending functor} if it satisfies the following two conditions:
		\begin{enumerate}
			\item $F^{2}=F$.
			\item There exists a natural transformation $t \co F \Rightarrow 1_{\A}$.
		\end{enumerate}
	Ascending functors are defined dually.
	\end{definition}

	A down beat object determines a descending functor and, dually, an up beat object determines an ascending functor.

	\begin{proposition}\label{closure_adjoint}
	Let $F \co \A \to \A$ be a descending functor on an acyclic category $\A$.
The functor $F \co \A \to F(\A)$ onto the category of the image $F$ is right adjoint to the inclusion functor.
	\end{proposition}
	\begin{proof}
	For a pair of objects $x \in \ob(\A)$ and $Fy \in \ob(F(\A))$, we can verify that the map
	\[
	F \co \A(Fy,x) \to F(\A)(F^2y,Fx)=F(\A)(Fy,Fx)
	\]
	is a natural isomorphism. Indeed, an inverse map is given by composing $t_x \co Fx \to x$:
	\[
	(t_x)_* \co F(\A)(Fy,Fx) \to \A(Fy,x).
	\]
	\end{proof}

	For an acyclic category $\A$ and an object $x$ of $\A$, let $\A \bs x$ denote the full subcategory of $\A$ with the set of objects $\ob(\A) \bs \{x\}$. The next corollary follows from Proposition \ref{closure_adjoint} and the functors associated with beat morphisms.

	\begin{corollary}\label{point_reduction}
	For a beat object $x$ in an acyclic category $\A$, the category $\A$ is strongly equivalent to $\A \bs x$.
	\end{corollary}

	\begin{remark}
	The definition of beat objects in Definition \ref{beat_object} can be extended to general categories with loops of morphisms. However, if we attempt to relate the strong homotopy theory to the reduction of beat objects, loops of morphisms are difficult to treat.
	For example, let us consider a nontrivial group $G$ regarded as a category with a single object and invertible endomorphisms. For a morphism (element) $g$ of $G$, any morphism $h$ is uniquely factored through $g$ as $h=g(g^{-1}h)$. The unique object becomes a beat object. However, $G$ and $G \bs * = \emptyset$ are not strong homotopy equivalent. 
	\end{remark}

	\begin{definition}\label{acyclic_collapse}
	For an acyclic category $\A$ with a beat object $x$, we say that there is an {\em elementary strong collapse} from $\A$ to $\A \bs x$. For a full subcategory $\B$ of $\A$, if there exists a sequence of elementary strong collapses starting in $\A$ and ending in $\B$, we say that there is a {\em strong collapse} from $\A$ to $\B$ and use the notation $\A \ses \B$. In particular, when $\B$ is the trivial category with a single object and the identity, we call $\A$ {\em strongly collapsible} and use the notation $\A \ses *$. 
	For two acyclic categories $\A$ and $\B$, we say that $\A$ and $\B$ {\em have the same strong equivalence type} if there exists a sequence of acyclic categories 
\[
\A = \C_0, \C_1, \ldots, \C_n = \B
\]
such that $\C_i \ses \C_{i+1}$ or $\C_{i+1} \ses \C_i$ for each $i$.
	\end{definition}

	\begin{definition}
	An acyclic category $\A$ is called {\em minimal} if it has no beat object.
	A full subcategory $\B$ of $\A$ is called a {\em core} if it is minimal and $\A \ses \B$.
	\end{definition}

 For an acyclic category $\A$, let $P(\A)$ denote the poset $\ob(\A)$ with the partial order $x \leq y$ that is given by $\A(x,y) \ne \emptyset$. We note that $\A$ and $P(\A)$ do not have the same strong equivalence type, as we can see by the next example.

	\begin{example}\label{S^1}
	Let $\cS^1$ be an acyclic category with two objects $x,y$ and two parallel morphisms between them:
		
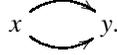
\[
		\xymatrix{
			x \ar@/^2ex/[r]^{} \ar@/_2ex/[r]_{} & y.
		}
		\]
	The associated poset $P(\cS^1)$ is described as $x<y$. We notice that $\cS^1$ is minimal. However, $P(\cS^1)$ is strongly collapsible.
	\end{example}

	Every acyclic category has a core by successively removing all beat objects.
	We will next show that a core is determined uniquely up to isomorphism.

	\begin{proposition}\label{F=1}
	Let $\A$ be a minimal acyclic category. If a functor $F \co \A \to \A$ is homotopic to the identity functor $1_\A$, then $F=1_A$.
	\end{proposition}
	\begin{proof}
	Without loss of generality, we may assume that there is a natural transformation $t \co F \Rightarrow 1_{\A}$.
	It suffices to show that $F$ is the identity on objects because of the acyclicity of $\A$.
	For a minimal object $x$ of the associated poset $P(\A)$, the object $Fx$ must be $x$ by the minimality of $x$. 
	We assume that $Fy=y$ for any object $y<x$ in $P(\A)$ and we will show that $Fx=x$. In the case $Fx \neq x$, we will prove that $x$ is a down beat object with the down beat morphism $t_x \co Fx \to x$. 
	For any morphism $g \co y \to x$, the morphism $Fg \co y=Fy \to Fx$ satisfies $t_x \circ Fg = g \circ t_y = g$. 
	It remains to show the universality of $Fg$. Let us consider a morphism $h \co y \to Fx$ with $t_x \circ h=g$.
	We have 
	\[
		Fg = F(t_x \circ h) = F(t_x) \circ Fh = \id_{Fx} \circ h = h.
	\]
	This contradicts the minimality of $\A$. Hence, $Fx=x$ for any object $x$ of $\A$. It implies that $t_x=\id_x$ for any object $x$ and $Ff=f$ for any morphism $f$ of $\A$. Hence, $F=1_{\A}$.
	\end{proof}

	\begin{corollary}\label{unique_core}
	For an acyclic category $\A$, a core is uniquely determined up to isomorphism.
	\end{corollary}
	\begin{proof}
	Let $\A_1$ and $\A_2$ be cores of $\A$. We have a strong collapse $\A \ses \A_{i}$ for each $i=1,2$. By Corollary \ref{point_reduction}, we have the deformation retraction $R_i \co \A \to \A_i$ with a sequence of natural transformations between $J_i \circ R_i$ and $1_{\A}$, where $J_i \co \A_i \to \A$ is the inclusion functor.
	The composition $R_2 \circ J_1 \co \A_1 \to \A_2$ has a homotopy inverse $R_1 \circ J_2$. Proposition \ref{F=1} implies that $\A_1$ and $\A_2$ are isomorphic due to their minimality.
	\end{proof}

	We show that two acyclic categories are strong homotopy equivalent (Definition \ref{acyclic_homotopy}) if and only if they have the same strong equivalence type (Definition \ref{acyclic_collapse}).

	\begin{lemma}\label{iso_acyclic}
	Isomorphic acyclic categories have the same strong equivalence type.
	\end{lemma}
	\begin{proof}
	Let $F \co \A \to \B$ be an isomorphism between acyclic categories. We choose a total order (linear extension) on $\ob(A)=\{a_1,\ldots,a_m\}$ satisfying $\A(a_j,a_k) = \emptyset$ whenever $k<j$. We define acyclic categories $\A_i$ as follows: 
the set object is 
	\[
		\ob(\A_i)=\{Fa_1, \ldots, Fa_i, a_{i+1}, \ldots,a_m\}
	\]
and the set of morphisms is 
	\[
		\A_i(x,y) = 
		\begin{cases}
			\A(x,y) & \textrm{if $x,y \in \ob(\A)$}, \\
			\B(x,y) & \textrm{if $x,y \in \ob(\B)$}, \\
			\{f' \co Fa_k \to a_{j} \mid f \co a_k \to a_j \} & \textrm{if $x=Fa_k, y=a_j, k \leq i < j$},\\
			\emptyset & \textrm{otherwise}.
		\end{cases}
	\]
	The composition can be naturally defined by $\A$, $\B$, and the isomorphism $F$. Moreover, we consider another acyclic category $\C_i$ as follows: the set of objects is 
	\[
		\ob(\C_i)=\{Fa_1, \ldots, Fa_i, a_i, a_{i+1}, \ldots,a_m\}
	\]
and the set of morphisms is 
	\[
		\C_i(x,y) = 
		\begin{cases}
			\A_i(x,y) & \textrm{if } x,y \in \ob(\A_i), \\
			\A_{i-1}(x,y) & \textrm{if } x,y \in \ob(\A_{i-1}), \\
			\{k\} & \textrm{if } x=Fa_i,y=a_{i},\\
			\emptyset & \textrm{otherwise}.
		\end{cases}
	\]
	The composition is given by $f \circ k = f'$ in $\A_i$ and $k \circ g=g'$ in $\A_{i-1}$.
	We notice that $a_i$ and $Fa_i$ are beat objects in $\C_i$, which associate $k$ as the beat morphism. Thus, $C_i \ses \A_i$ and $\C_i \ses \A_{i-1}$. Therefore, $A_{i-1}$ and $A_{i}$ have the same strong equivalence type for each $i$. Moreover, $\A=\A_0$ and $\B=\A_m$ have the same homotopy type as well.
	\end{proof}

	\begin{corollary}\label{acyclic_simeq_strong}
	Two acyclic categories $\A$ and $\B$ have the same strong equivalence type if and only if $\A \simeq \B$.
	\end{corollary}
	\begin{proof}
	If $\A$ and $\B$ have the same strong equivalence type, there is a sequence of elementary strong collapses starting in $\A$ and ending in $\B$. An elementary strong collapse induces a strong homotopy equivalence. Therefore, $\A \simeq \B$. 
	Conversely, we assume that $\A \simeq \B$. The cores $\A_0$ and $\B_0$ are strongly homotopy equivalent as well. Proposition \ref{F=1} implies that $\A_0 \cong \B_0$ and Lemma \ref{iso_acyclic} implies that these have the same strong equivalence type. 
	Since the original category and its core have the same strong equivalence type, $\A$ and $\B$ have the same strong equivalence type.
	\end{proof}


\section{Strong homotopy type of $\Delta$-complexes}

	In the previous section, we discussed the strong homotopy theory of acyclic categories as a generalization of posets (finite spaces). A poset can be associated with a geometric object that is called the {\em order complex}. This is a simplicial complex consisting of ordered sequences of elements of the poset. 

	For a general category, the {\em classifying space} (see Section 4) is known as a generalization of order complexes. The classifying space of an acyclic category is a space with triangulation. However, it is not a simplicial complex in general. 
	We formulate such triangulated spaces in purely combinatorial terms as generalizations of simplicial complexes. A similar definition in terms of regular cell complexes can be found in Appendix of Hatcher's book \cite{Hat02}.

	For a set $S$, let $2^{S}$ denote the power set of $S$. Moreover, let $2^{S}_{n}$ denote the subset of $2^{S}$ that consists of all subsets with cardinality $n+1$.

	\begin{definition}\label{Delta_complex}
	A {\em $\Delta$-complex} $X$ consists of the following data:
	\begin{itemize}
		\item A family of sets $\{X_{n}\}_{n \geq 0}$ indexed by nonnegative integers.
		\item A map $V \co X_n \to 2^{X_0}_n$ for each $n$.
		\item For each $v \in X_0$ and $n \geq 0$, a map $d_v \co X_n(v) \to X_{n-1}$, where $X_n(v)=V^{-1}(2^{X_0}_{n}(v))$ and $2^{X_0}_n(v) = \{A \in 2^{X_0}_n \mid v \in A\}$.
	\end{itemize}
	We call an element $\sigma$ in $X_n$ an {\em $n$-simplex} of $X$ and write $\dim \sigma =n$ for its dimension. In particular, we call an element in $X_0$ (resp., $X_1$) a {\em vertex} (resp., an {\em edge}) of $X$.
	These data are required to satisfy the following two conditions:
		\begin{enumerate}
			\item For each vertex $v \in X_0$ and $n \geq 1$, the following diagram is commutative:
				\[
				\xymatrix{
				X_{n}(v) \ar[r]^-{V} \ar[d]_{d_v}& 2^{X_0}_n(v) \ar[d]^{\hat{v}} \\
				X_{n-1} \ar[r]_{V} & 2^{X_0}_{n-1},
				}
				\]
		where $\hat{v} \co 2^{X_0}_n(v) \to  2^{X_0}_{n-1}$ deletes the vertex $v$.
			\item For two vertices $v,w \in X_0$ and $n \geq 2$, the following diagram is commutative:
				\[
				\xymatrix{
				X_{n}(v) \cap X_{n}(w) \ar[r]^-{d_v} \ar[d]_{d_w}& X_{n-1}(w) \ar[d]^{d_w} \\
				X_{n-1}(v) \ar[r]_{d_v} & X_{n-2}.
				}
				\]
		\end{enumerate}
	\end{definition}

	\begin{definition}
	For a $\Delta$-complex $X$, let $\chi(X)$ denote the set of all simplices of $X$.
	The face relation with co-dimension $1$ on $\chi(X)$ is denoted by $\sigma \prec_1 \tau$ and defined by $d_v(\tau)=\sigma$ for some $v \in V(\tau)$. The partial order generated from $\prec_1$ is denoted by $\preceq$ and $\chi(X)$ with $\preceq$ is called the {\em face poset} of $X$.
	\end{definition}

	For a vertex $v$ and $n$-simplex $\sigma$ of a $\Delta$-complex $X$, $v \preceq \sigma$ if and only if $v \in V(\sigma)$ by condition (1) in Definition \ref{Delta_complex}. We note that, unlike a simplicial complex, a simplex of a $\Delta$-complex is not uniquely determined by its vertices. That is, even if $V(\sigma)=V(\sigma')$, it may happen that $\sigma \neq \sigma'$. However, if such simplices have a common coface ($\tau$ with $\sigma,\sigma' \prec \tau$), then $\sigma=\sigma'$.

	\begin{example}
	A simplicial complex can be characterized as a $\Delta$-complex with injective map of vertices $V \co X_n \to 2_{n}^{X_0}$ for any $n \geq 0$.
	\end{example}

	Let $X$ be a $\Delta$-complex. A {\em subcomplex} $Y$ of $X$ is a $\Delta$-complex such that $Y_n \subset X_n$ and $V \co Y_n \to 2^{Y_0}_n$ and $d_v \co Y_n(v) \to Y_{n-1}$ are the restrictions of the maps that are associated with $X$ for each $n \geq 0$ and $v \in Y_0$.
	For a vertex $v$ of a $\Delta$-complex $X$, let $X \bs v$ be the subcomplex that is obtained by deleting the simplices $\sigma$ such that $v \preceq \sigma$. 

	\begin{definition}\label{Delta_map}
	For two $\Delta$-complexes $X$ and $Y$, a {\em $\Delta$-map} $(k,f) \co X \to Y$ consists of the following data: For each $\sigma \in X_{n}$, we have a pair $(k(\sigma),f(\sigma))$, where
		\begin{itemize}
			\item $k(\sigma) \leq n$ is a nonnegative integer,
			\item $f(\sigma) \in Y_{k(\sigma)}$ is a simplex of $Y$ with dimension $k(\sigma)$.
		\end{itemize}
		These data are required to satisfy the following conditions for any $\sigma \in X_n$:
		\begin{enumerate}
			\item $\tilde{f} (V(\sigma)) = V(f(\sigma))$ in $2^{Y_0}$, 
where the map $\tilde{f} \co 2^{X_0} \to 2^{Y_0}$ is induced from $f \co X_0 \to Y_0$. We note that $k(v)=0$ for any vertex $v \in X_0$.
			\item $d_{f(w)}(f(\sigma)) = f(d_w(\sigma))$ in $Y_{k(\sigma)}$ for each $w \in V(\sigma)$. We note that if $w \in V(\sigma)$, then $f(w) \in V(f(\sigma))=\tilde{f}(V(\sigma))$ by condition (1).
		\end{enumerate}
	\end{definition}

	The above definition is inspired by trisp-maps, which were introduced in Definition 2.48 of \cite{Koz08}. We write $f$ instead of $(k,f)$ when there is no confusion about $k$.
	We can easily verify that the composition of $\Delta$-maps is a $\Delta$-map and the identity $1_{X} \co X \to X$ is obviously a $\Delta$-map. Thus, $\Delta$-complexes and $\Delta$-maps constitute a category. A simplicial map between simplicial complexes is a $\Delta$-map. 
	For two $\Delta$-maps, we define a relation between them as a simple generalization of {\em contiguity} on simplicial maps \cite{Spa66}.

	\begin{definition}\label{contiguity}
	Let $(k,f)$ and $(\ell,g)$ be two $\Delta$-maps from $X$ to $Y$.
	We call these maps {\em contiguous} and use the notation $(k,f) \sim_c (\ell,g)$ if for any simplex $\sigma$ of $X$, there exists a unique simplex $\tau$ of $Y$ such that $V(\tau)=V(f(\sigma)) \cup V(g(\sigma))$ and $f(\sigma), g(\sigma) \preceq \tau$. The equivalence relation generated from $\sim_c$ is denoted by $\sim$.
When $(k,f) \sim (\ell,g)$, we say that these maps lie in the same contiguous class.

	A $\Delta$-map $(k,f) \co X \to Y$ is called {\em strong equivalence} if there exists $(\ell,g) \co Y \to X$ such that $(\ell,g) \circ (k,f) \sim 1_{X}$ and $(k,f) \circ (\ell,g) \sim 1_{Y}$. In this case, we say that $X$ and $Y$ are {\em strongly equivalent} and use the notation $X \sim Y$.
	\end{definition}

	\begin{definition}\label{dominated}
	Let $X$ be a $\Delta$-complex. A vertex $v \in X_0$ is {\em dominated} by another vertex $v'$ if for any simplex $\sigma \in X$ with $v \preceq \sigma$, there exists a unique simplex $\tau \in X$ with $V(\tau)=V(\sigma) \cup \{v'\}$ and $\sigma \preceq \tau$. By definition, there exists a unique edge $e$ spanning $v$ and $v'$. We call this the {\em dominated edge} associated with the dominated vertex $v$.

	When $v$ is a dominated vertex of $X$, we say that there is an {\em elementary strong collapse} from $X$ to $X \bs v$. For a subcomplex $Y$ of $X$, if there exists a sequence of elementary strong collapses starting in $X$ and ending in $Y$, we say that there is a {\em strong collapse} from $X$ to $Y$ and use the notation $X \ses Y$. In particular, when $Y$ consists of a single vertex, we call $X$ {\em strongly collapsible} and use the notation $X \ses *$. 
	For two $\Delta$-complexes $X$ and $Y$, we say that $X$ and $Y$ {\em have the same strong homotopy type} if there exists a sequence of $\Delta$-complexes 
\[
X = Z_0, Z_1, \ldots, Z_n = Y
\]
such that $Z_i \ses Z_{i+1}$ or $Z_{i+1} \ses Z_i$.
	\end{definition}

	Let $v$ be dominated by $v'$ in a $\Delta$-complex in $X$. For a simplex $\sigma$ including both $v$ and $v'$, we can find a simplex satisfying the above universality, that is, $\sigma$ itself. Hence, in order to prove that $v$ is dominated by $v'$, it is enough to consider simplices $\sigma$ that include $v$ and not $v'$.

	\begin{remark}\label{maximal_delta_simplicial}
	A vertex of a simplicial complex is dominated if the link forms a simplicial cone. In other words, a vertex $v$ is dominated by $v'$ if and only if every maximal simplex that contains $v$ contains $v'$ as well.
	However, in our setting of $\Delta$-complexes, this definition does not correspond to Definition \ref{dominated}. On direction is true. Namely, for a vertex $v$ that is dominated by $v'$, every maximal simplex that contains $v$ contains $v'$ as well. However, the converse is not true.
	For example, we may consider a circle $S^1$ with two vertices $v$, $v'$ and two edges. Every maximal simplex (edge) contains both $v$ and $v'$. However, neither $v$ nor $v'$ is dominated in the sense of Definition \ref{dominated}.
	\end{remark}

	A dominated vertex $v$ of a $\Delta$-complex determines a $\Delta$-map $r_v \co X \to X \bs v$ as follows: 
	\[
	r_v(\sigma)=
		\begin{cases}
			\sigma & \textrm{if $v \not \preceq \sigma$}, \\ \notag
			d_v \tau & \textrm{if $v \preceq \sigma$},
		\end{cases}
	\]
	where $\tau$ is the unique simplex with $V(\tau)=V(\sigma) \cup \{v'\}$ and $\sigma \preceq \tau$. We can easily verify that $r_v$ satisfies the conditions of $\Delta$-maps in Definition \ref{Delta_map}.

	\begin{proposition}\label{deformation_delta}
	The composition of the $\Delta$-map $r_v \co X \to X \bs v$ associated with a dominated vertex $v$ of a $\Delta$-complex $X$ and the inclusion $X \bs v \hookrightarrow X$ is contiguous to the identity map on $X$.
	\end{proposition}
	\begin{proof}
	We assume that $v$ is dominated by $v'$.
	For any simplex $\sigma \in X_{n}$, we have $r_v(\sigma)=\sigma$ when $v \preceq \sigma$, otherwise we have a unique simplex $\tau$ such that $V(\tau)=V(\sigma) \cup \{v'\} = V(\sigma) \cup V(r_v(\sigma))$ and contains $\sigma$ and $r_v(\sigma)=d_v(\tau)$ as faces. Hence, the composition of $r_v$ and the inclusion is contiguous to the identity.
	\end{proof}

	A $\Delta$-complex provides the gluing data of simplices. We can construct a topological space according to these data.

	\begin{definition}
	Let $X$ be a $\Delta$-complex. 
	We choose a vector $x_v$ in $\mathbb{R}^{N}$ for each $v \in X_0$ and sufficiently large $N>0$ such that the points $\{x_v\}_{v \in X_0}$ lie in general position. Let $\Delta(\sigma)$ denote the convex hull of the vertices $\{x_v\}_{v \in V(\sigma)}$ for a simplex $\sigma$ of $X$. The face of $\Delta(\sigma)$ not including a vertex $w \in V(\sigma)$ is defined as the subspace
	\[
		d_w(\Delta(\sigma)) = \left\{ \sum_{v  \in V(\sigma)} t_v x_v \in \Delta(\sigma) \mid t_w=0\right\}.
	\]
	The (geometric) realization of $X$ is denoted by $|X|$ and defined by
	\[
		|X| = \left(\coprod_{\sigma \in X} \Delta(\sigma) \right) \mathord{/} d_w(\Delta(\sigma)) \mathord{\sim} \Delta(d_w(\sigma)).
	\]
	\end{definition}

	The realization of a $\Delta$-complex $X$ is equipped with the natural characteristic map 
	\[
		\varphi_{\sigma} \co \Delta^{n} \to \Delta(\sigma) \hookrightarrow |X|
	\]
	for each $n$-simplex $\sigma$ of $X$. Every characteristic map is a homeomorphism onto its image. Hence, $|X|$ is a regular cell complex. 
	This topological formulation of $\Delta$-complexes can be found in Appendix of Hatcher's book \cite{Hat02}.

	For a regular cell complex, the simple homotopy theory based on simple collapses is a fundamental tool in combinatorial algebraic topology.  A cell $\sigma$ of a regular cell complex $X$ is a {\em free face} if it has a unique coface $\tau$ with co-dimension $1$.
We note that we have a deformation retraction $X \to X \bs \{\sigma \cup \tau\}$ in this case. We denote it by $X \searrow X \bs \{\sigma \cup \tau\}$ and call it an {\em elementary simple collapse}. 
A sequence of elementary simple collapses is called a {\em simple collapse}. Compared with strong collapses based on removing vertices, simple collapses are based on removing pairs of cells.

	\begin{lemma}
	If $X \ses Y$, then $|X| \searrow |Y|$.
	\end{lemma}
	\begin{proof}
	We may assume that $Y=X\bs v$ for a dominated vertex $v$ of $X$. Furthermore, we assume that $v$ is dominated by $v'$. 
	Any maximal simplex $\tau$ that contains $v$ contains $v'$ as well. Let $\sigma$ be the face $d_{v'}(\tau)$ of $\tau$. Since $v$ is dominated, $\tau$ is the unique simplex such that $V(\tau)=V(\sigma) \cup \{v'\}$ and $\sigma \prec \tau$. We will show that $\sigma$ is a free face of $\tau$.
	We assume that $\sigma$ has a coface $\tau'$, i.e., $\sigma \prec \tau'$. We have a simplex $\rho$ such that $V(\rho)=V(\tau') \cup \{v'\}$ and $\tau' \preceq \rho$. The universality of $\tau$ implies that $\rho$ contains $\tau$ as a face. Moreover, the maximality of $\tau$ implies that $\rho=\tau$. The order relations $\sigma \prec \tau' \preceq \tau$ and $\dim \tau= (\dim \sigma) +1$ state that $\tau=\tau'$.
Thus, $\sigma$ is a free face of $\tau$ and we obtain an elementary simple collapse $|X| \searrow |X_1|=|X \bs \{\sigma, \tau\}|$. 
	We notice that $v$ remains to be dominated by $v'$ in $X_1$. If we iterate this process, then the last step is a collapse of the dominated edge $e$, which is the minimal simplex containing $v$ and $v'$. Thus, we have a sequence of elementary simple collapses:
	\[
		|X| \searrow |X_1| \searrow \ldots \searrow |X_n|=|X \bs v| \cup_{v'} |e| \searrow |X \bs v|.
	\]
	\end{proof}

	\begin{definition}
	A $\Delta$-complex $X$ is called {\em minimal} if it has no dominated vertex. A $\Delta$-subcomplex $Y$ of $X$ is called a {\em core} if it is minimal and $X \ses Y$.
	\end{definition}

	\begin{proposition}\label{unique_minimal_delta}
	Let $X$ be a minimal $\Delta$-complex. If $f \co  X \to X$ is contiguous to the identity map $1_X$, then $f=1_X$.
	\end{proposition}
	\begin{proof}
	For a vertex $v \in X_0$, we assume that $f(v) \neq v$. We will show that $v$ is dominated by $f(v)$. For any simplex $\sigma \in X_n$ with $v \preceq \sigma$, there exists a unique simplex $\tau$ such that $V(\tau)=V(f(\sigma)) \cup V(\sigma)$ and $f(\sigma),\sigma \preceq \tau$. Since $f(v) \preceq f(\sigma)$, we obtain the face $\tau_{f(v)}$ of $\tau$ with $V(\tau_{f(v)})=V(\sigma) \cup \{f(v)\}$ and $\sigma \preceq \tau_{f(v)}$. 
	It remains to show the universality of $\tau_{f(v)}$. We assume that we have a simplex $\tau'$ with $V(\tau')=V(\sigma) \cup \{f(v)\}$ and $\sigma \preceq \tau'$.
	By contiguity, there is a unique simplex $\rho$ such that $V(\rho)=V(f(\tau')) \cup V(\tau')$ and $\tau', f(\tau') \preceq \rho$.
	We notice that $\sigma \preceq \tau' \preceq \rho$ and $f(\sigma) \preceq f(\tau') \preceq \rho$ and the universality of $\tau$ ensures that $\tau \preceq \rho$. Since $\tau_{f(v)}$ and $\tau'$ have a common coface $\rho$, they must be equal.
	Thus, $v$ is dominated by $f(v)$. This contradicts the minimality of $X$.
	Therefore, $f(v)=v$ for every vertex $v$. Moreover, for a general simplex $\sigma$ of $X$, by contiguity and the fact that  $V(\sigma)=V(f(\sigma))$, we conclude that $f(\sigma)=\sigma$.
	\end{proof}

	The next corollary can be proved as Corollary \ref{unique_core}.

	\begin{corollary}\label{unique_core_delta}
	For a $\Delta$-complex $X$, a core is uniquely determined up to isomorphism.
	\end{corollary}

	\begin{lemma}\label{iso_delta}
	Isomorphic $\Delta$-complexes have the same strong homotopy type.
	\end{lemma}
	\begin{proof}
	Let $\varphi \co X \to Y$ be an isomorphism between $\Delta$-complexes. We choose a total order on vertices $X_0=\{v_1,\ldots,v_m\}$, and define a $\Delta$-complex $X(i)$ as follows: the set of vertices is 
	\[
		X(i)_0=\{\varphi(v_1), \ldots, \varphi(v_i),v_{i+1},\ldots,v_m\}
	\] 
	and the set of $n$-simplices is
	\begin{align*}
		X(i)_n &= \{ \tau \in Y_n \mid V(\tau) \subset \{\varphi(v_k)\}_{k=1}^{i}\} \\
		&\cup \{ \sigma \in X_n \mid V(\sigma) \subset \{v_k\}_{k=i+1}^{m}\} \\
		&\cup \{\sigma' \mid \sigma \in X_n, v_j, v_k \in V(\sigma) \textrm{ for some } j \leq i <k \},
	\end{align*}
where $V(\sigma') = \{\varphi(v_k)\}_{v_k \preceq \sigma, k \leq i} \cup \{v_k\}_{v_k \preceq \sigma, i<k}$ and
	\[
		d_w(\sigma') = 
		\begin{cases}
			\sigma & \textrm{if $w$ is the unique vertex of $Y$ contained in $\sigma'$}, \\
			\varphi(\sigma) & \textrm{if $w$ is the unique vertex of $X$ contained in $\sigma'$}, \\ 
			(d_w(\sigma))' & \textrm{otherwise}. 
		\end{cases}
	\]
	We have the canonical isomorphism $\varphi_i \co X(i) \to X$ by $\varphi$ for each $i$. Moreover, we consider another $\Delta$-complex $Z(i)$ as follows: the set of vertices is 
	\[
		Z(i)_0=\{\varphi(v_1), \ldots, \varphi(v_i),v_i,v_{i+1},\ldots,v_m\}
	\]
and the set of $n$-simplices is
	\[
		Z(i)_n = X(i)_n \cup X(i-1)_n \cup \{\varphi(v_i)\sigma \mid \sigma \in X(i-1)_{n-1}, v_i \preceq \sigma\},
	\]
where $V(\varphi(v_i)\sigma) = V(\sigma) \cup \{\varphi(v_i)\}$ and 
	\[
		d_w(\varphi(v_i)\sigma) = 
		\begin{cases}
			\sigma & \textrm{if $w=f(v_i)$}, \\
			\varphi_i^{-1}(\varphi_{i-1}(\sigma)) & \textrm{if $w=v_i$}, \\
			\varphi(v_i)d_w(\sigma) & \textrm{otherwise}.
		\end{cases}
	\]
	We can easily verify that $v$ is dominated by $\varphi(v)$ in $Z(i)$. Conversely, $\varphi(v_i)$ is dominated by $v_i$. For any simplex $\sigma$ in $Z(i)_n$ with $\varphi(v_i) \preceq \sigma$ and $v_i \not\preceq \sigma$, it belongs to $X(i)_n$.
	The simplex $\tau=\varphi_{i-1}^{-1}\varphi_i(\sigma)$ in $X(i-1)_n$ contains $v_i$ and $\varphi(v_i)\tau$, and is the unique simplex such that $V(\varphi(v_i)\tau)= V(\tau) \cup \{\varphi(v_i)\} = V(\sigma) \cup \{v_i\}$ and $d_{v_i}(\varphi(v_i)\tau)=\sigma$. 
	Thus, $Z(i) \ses X(i)$ and $Z(i) \ses X(i-1)$. This implies that $X(i)$ and $X(i-1)$ have the same strong homotopy type for each $i$. Moreover, $X=X(0)$ and $Y=X(m)$ have the same strong homotopy type as well.
	\end{proof}

	The following corollary is proved  by an argument similar to that of Corollary \ref{acyclic_simeq_strong} and Lemma \ref{iso_delta}.

	\begin{corollary}
	Two $\Delta$-complexes $X$ and $Y$ have the same strong homotopy type if and only if $X \sim Y$.
	\end{corollary}


\section{Strong homotopy type of classifying spaces and face posets}

\subsection{Classifying spaces of acyclic categories}

	The {\em classifying space} $B\A$ of an acyclic category $\A$ is a $\Delta$-complex defined as follows:
	The set of $n$-simplices $B\A_n$ consists of sequences of composable $n$-morphisms that are formed of
	\[
		\sigma = a_0 \stackrel{f_1}{\to} a_1 \stackrel{f_2}{\to} \ldots \stackrel{f_{n}}{\to} a_n.
	\]
not including identities. The map of vertices is given by $V(\sigma)=\{a_0, \ldots,a_n\}$ and the face map $d_x$ is given by
	\[
		d_x(\sigma) =
		\begin{cases}
			a_1 \stackrel{f_2}{\to} \ldots \stackrel{f_{n}}{\to} a_n & \textrm{if } x=a_0, \\
			a_0 \stackrel{f_1}{\to} \ldots \to a_{i-1} \stackrel{f_{i+1}\circ f_i}{\to} a_{i+1} \to \ldots \stackrel{f_{n}}{\to} a_n &  \textrm{if } x=a_i, 1 \leq i \leq n-1\\
			a_0 \stackrel{f_1}{\to} \ldots \stackrel{f_{n-1}}{\to} a_{n-1} & \textrm{if } x=a_n. \\
		\end{cases}
	\]
	The classifying space gives rise to a functor from the category of acyclic categories to the category of $\Delta$-complexes. For a functor $F \co \A \to \B$ between acyclic categories, the $\Delta$-map $BF \co B\A \to B\B$ sends an $n$-simplex $\sigma$ of $B\A$ formed of the above to a simplex of $B\B$
	\[
		R(Fa_0 \stackrel{Ff_1}{\to} Fa_1 \stackrel{Ff_2}{\to} \ldots \stackrel{Ff_{n}}{\to} Fa_n),
	\]
where $R$ removes identities.

	\begin{proposition}\label{beat_dominate}
	Let $\A$ be an acyclic category. A beat object $x$ of $\A$ is a dominated vertex of $B\A$.
	\end{proposition}
	\begin{proof}
	Let $f \co y \to x$ be the down beat morphism associated with a down beat object $x$. For any $n$-simplex 
	\[
		\sigma = a_0 \stackrel{g_1}{\to} \ldots \stackrel{g_{i-1}}{\to} a_{i-1} \stackrel{g_i}{\to} x \stackrel{g_{i+1}}{\to} a_{i+1} \stackrel{g_{i+2}}{\to} \ldots \stackrel{g_{n}}{\to} a_n
	\]
in $B\A$ not including $y$ as a vertex, we obtain an $(n+1)$-simplex
	\[
		\tau = a_0 \stackrel{g_1}{\to} \ldots \stackrel{g_{i-1}}{\to} a_{i-1} \stackrel{\tilde{g}_i}{\to} y \stackrel{f}{\to} x \stackrel{g_{i+1}}{\to} a_{i+1} \stackrel{g_{i+2}}{\to} \ldots \stackrel{g_{n}}{\to} a_n
	\]
satisfying $V(\tau)=V(\sigma) \cup \{y\}$ and $\sigma \preceq \tau$. The universality of $\tau$ follows from the property of the beat morphism $f$. This implies that $x$ is dominated by $y$ in the classifying space $B\A$. We can similarly show that an up beat object of $\A$ is a dominated vertex of $B\A$.
	\end{proof}

	The above proposition implies that if $\A \ses \A \bs x$, then $B\A \ses B\A \bs x = B(\A \bs x)$ for a beat object $x$ of $\A$. Repeating this process, we obtain the following result.

	\begin{proposition}\label{B>>}
	Let $\A$ be an acyclic category. For a full subcategory $\A_0$ of $\A$, if $\A \ses \A_0$, then $B\A \ses B\A_0$.
	\end{proposition}

	\begin{theorem}\label{B_minimal}
	An acyclic category $\A$ is minimal if and only if $B\A$ is minimal.
	\end{theorem}
	\begin{proof}
	If $\A$ is not minimal, then $B\A$ is not minimal, since the classifying space sends beat objects to dominated vertices by Proposition \ref{beat_dominate}.
	Conversely, we assume that $B\A$ is not minimal, i.e., it has a vertex $x$ dominated by $y$. There exists a unique edge spanning $x$ and $y$. It corresponds to a morphism $f$ of $\A$ between $x$ and $y$ and we may assume that $f \co y \to x$.
	We choose a morphism $f' \co y \to x'$ satisfying the following two conditions:
	\begin{itemize}
		\item $h \circ f' =f$ for some morphism $h \co x' \to x$.
		\item $f'$ is indecomposable, i.e., it can never be described as the composition of nontrivial morphisms.
		\item $f'$ is the only morphism between $y$ and $x'$.
	\end{itemize}
	That is, we take $f'$ to be a minimal element of the subposet of $P(y/\A)$ associated with the under category $y/\A$ of morphisms starting at $y$:  
	\[
		\{k \co y \to y' \in P(y/\A) \mid k \leq f, \A(y,y')=\{k\} \}.
	\]
	This poset includes $f$. Hence, we can find a minimal element. For any morphism $g \co z \to x'$, the composition $h \circ g \co z \to x$ associates a $1$-simplex of $B\A$ with $x \prec h \circ g$. 
Hence, there exists a $2$-simplex $\sigma$ of $B\A$ such that $V(\sigma)=(z,y,x)$ and $d_y(\sigma)= h \circ g \preceq \sigma$. We have two possibilities:
	\begin{enumerate}
		\item $\sigma = y \stackrel{i}{\to} z \stackrel{h \circ g}{\to} x$.
		\item $\sigma = z \stackrel{i}{\to} y \stackrel{f}{\to} x$ with $f \circ i=h \circ g$.
	\end{enumerate}
	Let us consider case (1). Since the morphism $f'$ is the unique morphism between $y$ and $x'$, we have $g \circ i=f'$. However, this contradicts the indecomposability of $f'$. Thus, we need only consider case (2). In order to prove that $x'$ is a down beat object, we will show that $g$ is uniquely factored through $f'$.
	For the $2$-simplex 
	\[
		\tau = z \stackrel{g}{\to} x' \stackrel{h}{\to} x,
	\]
in $B\A$ with $x \preceq \tau$, there exists a unique $3$-simplex $\upsilon$ with $V(\upsilon)=\{z,y,x',x\}$ and $d_y(\upsilon) =\tau$. The simplex $\upsilon$ can be described as follows:
	\[
		\xymatrix{
		z \ar[r]_{\tilde{g}} \ar@/^2ex/[rr]^{g} & y \ar[r]_{f'} & x' \ar[r]_{h} & x. 
		}
	\]
	The condition $d_y(\upsilon)=\tau$ implies that $f' \circ \tilde{g}=g$. 
	It remains to show the universality of $\tilde{g}$. If there exists $\tilde{g}'$ with $f' \circ \tilde{g}'=g$, we obtain a $3$-simplex $\upsilon'$ consists of $h$, $f'$ and $\tilde{g}'$. Then, $d_y \upsilon = \tau$ and $\upsilon=\upsilon'$ by the universality of $\upsilon$. We have $\tilde{g}=\tilde{g}'$. Thus, $x'$ is a down beat object of $\A$ and $\A$ is not minimal.
	\end{proof}

	\begin{theorem}\label{B_collapse}
	An acyclic category $\A$ is strongly collapsible if and only if the classifying space $B\A$ is strongly collapsible.
	\end{theorem}
	\begin{proof}
	If $\A \ses *$, then $B\A \ses *$ by Proposition \ref{B>>}. Conversely, we assume that $B\A \ses *$. If $\A_{0}$ denotes a core of $\A$, then $\A \ses \A_0$ and $B\A \ses B\A_0$.
	Theorem \ref{B_minimal} implies that $B\A_0$ is a core of $B\A$ and $B\A_0$ consists of a single vertex, since a core is uniquely determined, up to isomorphism, by Corollary \ref{unique_core_delta}. 
Thus, $\A_0$ consists of a single object with the identity, and $\A \ses *$.
	\end{proof}

\subsection{Face posets of $\Delta$-complexes}

	The face posets of $\Delta$-complexes give rise to a functor from the category of $\Delta$-complexes to the category of posets. 
	For a $\Delta$-map $(k,f) \co X \to Y$, the order-preserving map $\chi(k,f) \co \chi(X) \to \chi(Y)$ is defined to be $f$.

	\begin{proposition}\label{chi>>}
	Let $X$ be a $\Delta$-complex and $Y$ be a subcomplex of $X$. If $X \ses Y$, then $\chi(X) \ses \chi(Y)$.
	\end{proposition}
	\begin{proof}
	We assume that a vertex $v$ of $X$ is dominated by $v'$. It suffices to construct a deformation retraction $\chi(X) \to \chi(X \bs v)$ by Corollary 4.9 in \cite{BM12}. We have already constructed a deformation retraction $r_v \co X \to X \bs v$ after Remark \ref{maximal_delta_simplicial}. It induces the desired deformation retraction $\chi(r_v)$.
	\end{proof}

	By the above proposition, if a $\Delta$-complex $X$ has a dominated vertex, then $\chi(X)$ has a beat point. This implies the following corollary.

	\begin{corollary}\label{chi_minimal_delta}
	Let $X$ be a $\Delta$-complex. If $\chi(X)$ is minimal, then $X$ is minimal.
	\end{corollary}

	\begin{remark}
	Unfortunately, the converse of Corollary \ref{chi_minimal_delta} is not true, in contrast with the case of classifying spaces in Theorem \ref{B_minimal}. For example, Barmak and Minian obtained an example of a simplicial decomposition on a $2$-simplex that is collapsible and not strongly collapsible (see Figure 2 in \cite{BM12}). This simplicial complex $X$ is minimal, whereas $\chi(X)$ is not minimal, since $X$ has a free face that is a beat object of $\chi(X)$.
	\end{remark}

	An analog of Theorem \ref{B_collapse} for face posets and strong collapsibility is true. We will prove it after examining the properties of barycentric subdivision in the next subsection. 

	\begin{theorem}\label{chi_collapse}
	A $\Delta$-complex $X$ is strongly collapsible if and only if the face poset $\chi(X)$ is strongly collapsible.
	\end{theorem}

	\begin{proposition}\label{chi_homotopic}
	If two $\Delta$-maps $(k,f), (\ell,g) \co X \to Y$ lie in the same contiguous class, then $\chi(k,f)$ and $\chi(\ell,g)$ are homotopic.
	\end{proposition}
	\begin{proof}
	We define a map $h \co \chi(X) \to \chi(Y)$ by $h(\sigma)= \tau$, where $\tau$ is the unique simplex such that $V(\tau)=V(f(\sigma)) \cup V(g(\sigma))$ and $f(\sigma), g(\sigma) \preceq \tau$. For an ordered pair $\sigma \prec \sigma'$, we notice that $f(\sigma),g(\sigma) \preceq h(\sigma')$. The universality of $h(\sigma)$ implies that $h(\sigma) \preceq h(\sigma')$ and $h$ is order-preserving. 
	We have $\chi(k,f) \Rightarrow h \Leftarrow \chi(\ell,g)$ as functors. Hence, $\chi(k,f)$ and $\chi(\ell,g)$ are homotopic.
	\end{proof}

	\begin{remark}
	The analogous proposition of the above for classifying space is not true. It is true for finite posets as shown in Proposition 4.11 of \cite{BM12}. Let us consider $\cS^{0}$ consisting of two objects $x,y$ with only the identities and $\cS^1$ (which was introduced in Example \ref{S^1}) consisting of two object $x,y$ with two parallel morphisms $f,g$ between them. Let $F \co \cS^{0} \to \cS^{1}$ be the constant functor onto $x$ and let $G \co \cS^{0} \to \cS^{1}$ be the constant functor onto $y$.
	There is a natural transformation $t \co F \Rightarrow G$ defined by $f$. Hence, $F$ and $G$ are homotopic. However, $BF$ and $BG$ do not lie in the same contiguous class. There are only four $\Delta$-maps from $B\cS^{0}$ to $B\cS^1$ including $BF$ and $BG$ (the others are inclusion and switching).
$BF$ is not contiguous to any other $\Delta$-map, since $1$-simplices in $B\cS^1$ spanning $x$ and $y$ are not determined uniquely.
	\end{remark}

\subsection{Barycentric subdivisions of acyclic categories and $\Delta$-complexes}

	For an acyclic category $\A$, the barycentric subdivision $\sd(\A)$ is defined as the face poset $\chi(B\A)$ of the classifying space $B\A$. On the other hand, for a $\Delta$-complex $X$, the barycentric subdivision $\sd(X)$ is defined as the classifying space $B(\chi(X))$ of the face poset $\chi(X)$.
	The barycentric subdivision generally does not have the same strong homotopy type as the original acyclic category or $\Delta$-space. 
	However, strong collapsibility is compatible with the barycentric subdivision. We first focus on the strong collapsibility of a $\Delta$-complex $X$ and its barycentric subdivision $\sd(X)$. An $n$-simplex of $\sd(X)$ can be expressed as a sequence of simplices of $X$:
	\[
		\sigma_0 \prec \sigma_1 \prec \ldots \prec \sigma_n.
	\]
	For a simplex $\sigma$ of $X$, the barycenter $\hat{\sigma}$ is the vertex of $\sd(X)$ consisting of $\sigma$ itself.
	The next theorem can be proved as Theorem 4.15 in \cite{BM12}, but we need to pay attention to the universality of simplices.

	\begin{theorem}\label{sd_delta_collapse}
	Let $X$ be a $\Delta$-complex. Then, $X$ is strongly collapsible if and only if $\sd(X)$ is strongly collapsible.
	\end{theorem}
	\begin{proof}
	If $X \ses *$, then $\sd(X)=B\chi(X) \ses *$ by Proposition \ref{chi>>} and \ref{B>>}. We assume that $\sd(X) \ses *$ and $Y$ is a core of $X$. A strong collapse $X \ses Y$ induces $\sd(X) \ses \sd(Y)$. Hence, $\sd(Y)$ is strongly collapsible.
	There is a sequence of elementary strong collapses 
	\[
		\sd(Y)=Y(0) \ses Y(1) \ses \ldots \ses Y(n)=*,
	\]
where $Y(i+1) = Y(i) \bs w_i$ for some dominated vertex $w_i$ of $Y(i)$ and each $i$. We note that any $Y(i)$ is a full subcomplex of $\sd(Y)$ (every simplex $\sigma$ in $\sd(Y)$ with $V(\sigma) \subset Y(i)_0$ belongs to $Y(i)$).
	We will prove by induction on $i$ that for every maximal simplex $\sigma$ and vertex $v$ of $Y$, the barycenters $\hat{\sigma}$ and $\hat{v}$ belong to $Y(i)$ as vertices. 

	For a maximal simplex $\sigma$ in $Y$, the barycenter $\hat{\sigma}$ belongs to $Y(i)$ by induction. Assume that $\hat{\sigma}$ is dominated by another vertex $\hat{\tau}$ in $Y(i)$, where $\tau$ is a simplex of $Y$. 
	There is an edge $(\tau \prec \sigma)$ belongs to $Y(i)$ because of the maximality of $\sigma$. For any vertex $v \in V(\sigma)$, the barycenter $\hat{v}$ is a vertex of $Y(i)$ by induction. Since $\hat{\sigma}$ is dominated by $\hat{\tau}$, we have $v \prec \tau \prec \sigma$ in $Y$. Every vertex of $\sigma$ belongs to $\tau$, and it implies that $\tau \not\prec \sigma$. This is a contradiction. The barycenter $\hat{\sigma}$ is not dominated in $Y(i)$. It turns out that $\hat{\sigma} \in Y(i+1)$.

	We next focus on vertices. Assume that $\hat{v}$ is a vertex of $Y(i)$ for any vertex $v$ of $Y$ by induction.
	Moreover, we assume that $\hat{v}$ is dominated by another vertex $\hat{\tau}$ in $Y(i)$, where $\tau$ is a simplex of $Y$. There is an edge $(v \prec \tau)$ in $Y(i)$ and we take a vertex $v' \in V(\tau)$ with $v' \neq v$.
	In order to prove that $v$ is dominated by $v'$ in $Y$, we take a simplex $\sigma$ of $Y$ such that $v \preceq \sigma$ and $v' \not\preceq \sigma$.  There are two cases:
	\begin{enumerate}
		\item $\hat{\sigma}$ is a vertex of $Y(i)$.
		\item $\hat{\sigma}$ is not a vertex of $Y(i)$, i.e., it is dominated in $Y(j)$ for some $j<i$.
	\end{enumerate}
	We first consider case $(1)$. Since $\hat{v}$ is dominated by $\hat{\tau}$ in $Y(i)$ and $v' \not \preceq \sigma$, we have $v \prec \sigma \prec \tau$ in $Y$. We obtain a face $\tau_{v'}$ of $\tau$ such that $\sigma \prec \tau_{v'} \preceq \tau$ with $V(\tau_{v'})=V(\sigma) \cup \{v'\}$. It remains to show the universality of $\tau_{v'}$. We take a simplex $\tau'$ with $V(\tau')=V(\tau_{v'})$ and $\sigma \prec \tau'$, and choose a maximal simplex $\rho$ with $\tau' \prec \rho$. By induction, the barycenter $\hat{\rho}$ belongs to $Y(i)$. Since $\hat{v}$ is dominated by $\hat{\tau}$ and $(v \prec \rho)$ is an edge in $Y(i)$, we have the order relations $v \prec \tau \preceq \rho$ in $Y$. The simplices $\tau_{v'}$ and $\tau'$ have a common coface $\rho$. Hence, $\tau_{v'}=\tau'$.

	In case $(2)$, let $\hat{\sigma}$ be dominated by another vertex $\hat{\rho}$ in $Y(j)$, where $\rho$ is a simplex in $Y$. We have $\sigma \prec \rho$ in $Y$ since every vertex of $\sigma$ is contained in $\rho$.
	We can take the face $\rho_{v'}$ of $\rho$ with $V(\rho)=V(\sigma) \cup \{v'\}$ and $\sigma \prec \rho_{v'}$.
	It remains to show the universality of $\rho_{v'}$. We take a simplex $\rho'$ with $V(\rho')=V(\sigma) \cup \{v'\}$ and $\sigma \prec \rho'$ in $Y$, and choose a maximal simplex $\upsilon$ with $\rho' \preceq \upsilon$. By induction, the barycenter $\hat{\upsilon}$ belongs to $Y(j)$. Since $\hat{\sigma}$ is dominated by $\hat{\rho}$ and $(\sigma \prec \upsilon)$ is an edge in $Y(j)$, we have the order relations $\sigma \prec \rho \preceq \upsilon$ in $Y$. The simplices $\rho_{v'}$ and $\rho'$ have a common coface $\upsilon$. Hence, $\rho_{v'}=\rho'$.

	Both cases contradict the minimality of $Y$. Hence, the barycenter of every vertex of $Y$ belongs to $Y(i+1)$. Finally, $Y(n)=*$ contains $\hat{v}$ for all $v \in Y_0$ and $Y=*$.
	\end{proof}

	We are now ready to prove Theorem \ref{chi_collapse}.

	\begin{proof}[Proof of Theorem \ref{chi_collapse}] 
	If $X \ses *$, then $\chi(X) \ses *$ by Proposition \ref{chi>>}. If $\chi(X) \ses *$,  Proposition \ref{B>>} implies that $\sd(X)=B\chi(X) \ses *$ and $X \ses *$ follows from Theorem \ref{sd_delta_collapse}.
	\end{proof}

	Moreover, for the barycentric subdivision $\sd(\A)=\chi(\B(\A))$ of an acyclic category $\A$, the following result is an immediate consequence of Theorem \ref{B_collapse} and \ref{chi_collapse}.

	\begin{theorem}
	Let $\A$ be an acyclic category. Then, $\A$ is strongly collapsible if and only if $\sd(\A)$ is strongly collapsible.
	\end{theorem}


\begin{thebibliography}{AAA99}
		\bibitem[Bar11]{Bar11} J. A. Barmak. \textit{Algebraic topology of finite topological spaces and applications}. Lecture Notes in Mathematics, 2032. Springer, Heidelberg, 2011. xviii+170 pp.
		\bibitem[BM12]{BM12} J. Barmak; E. G. Minian. Strong homotopy types, nerves and collapses. \textit{Discrete Comput. Geom}. 47 (2012), no. 2, 301--328.
		\bibitem[EZ50]{EZ50} S. Eilenberg; J.A. Zilber. Semi-simplicial complexes and singular homology. \textit{Ann. of Math}. (2) 51, (1950). 499--513.
		\bibitem[Hat02]{Hat02} A. Hatcher. \textit{Algebraic topology}. Cambridge University Press, Cambridge, 2002. xii+544 pp.
		\bibitem[Hof74]{Hof74} G. Hoff. Cat\'egories fibr\'ees et homotopie. (French) \textit{C. R. Acad. Sci. Paris S\'er}. A 278 (1974), 223--225.
		\bibitem[Koz08]{Koz08} D. Kozlov. \textit{Combinatorial algebraic topology. Algorithms and Computation in Mathematics}, 21. Springer, Berlin, 2008. xx+389 pp.
		\bibitem[Lee73]{Lee73}M. J. Lee. Homotopy for functors. \textit{Proc. Amer. Math. Soc}. 36 (1972), 571--577; erratum, ibid. 42 (1973), 648--650.
		\bibitem[Min02]{Min02}E. G. Minian. $\mathbf{Cat}$ as a $\Lambda$-cofibration category. \textit{J. Pure Appl. Algebra} 167 (2002), no. 2-3, 301--314.
		\bibitem[RS71]{RS71} C. P. Rourke; B. J. Sanderson. $\Delta$-sets. I. \textit{Homotopy theory. Quart. J. Math. Oxford Ser}. (2) 22 (1971), 321--338.
		\bibitem[Spa66]{Spa66}E. H. Spanier. \textit{Algebraic topology}. Corrected reprint of the 1966 original. Springer-Verlag, New York, 1966. {\rm xvi}+528 pp.
		\bibitem[Sto66]{Sto66} R. E. Stong.  Finite topological spaces. \textit{Trans. Amer. Math. Soc}. 123 (1966), 325--340.
		
	\end{thebibliography}
\end{document}